\documentclass[reqno,11pt]{amsart}
\usepackage{latexsym,amsfonts,amssymb,amsmath,euscript}
\usepackage{graphicx}
\usepackage{hyperref}
\usepackage{color,fancybox}
\usepackage[normalem]{ulem}  
\usepackage{mathrsfs}

\usepackage{latexsym} 
\usepackage{comment}
\usepackage[british]{babel}

\usepackage{amssymb}
\numberwithin{equation}{section}
\newtheorem{theorem}{Theorem}[section]
\newtheorem{lemma}{Lemma}[section]
\newtheorem{proposition}{Proposition}[section]
\newtheorem{definition}{Definition}[section]
\newtheorem{remark}{Remark}[section]

\numberwithin{equation}{section}

\newcommand{\Img}{\mathop\mathrm{Im}\nolimits}
\newcommand{\Ker}{\mathop\mathrm{Ker}\nolimits}
\newcommand{\coKer}{\mathop\mathrm{coKer}\nolimits}

\newcommand{\dimm}{\mathop\mathrm{dim}\nolimits}
\newcommand{\codimm}{\mathop\mathrm{codim}\nolimits}

\title[Existence of periodic solution for a class of neutral differential equations with impulses]{Existence of periodic  solution for a class of neutral differential equations with impulses}

\author[S. M. Afonso]{Suzete M. Afonso}
\address[S. M. S. Afonso]{Universidade Estadual Paulista (UNESP), Instituto de Geoci\^{e}ncias e Ci\^{e}ncias Exatas, 13506-900, Rio Claro SP, Brasil}
\email{smafonso@rc.unesp.br}

\author[A. L. Furtado]{Andr\'e L. Furtado}
\address[A. L. Furtado]{Universidade do Estado do Rio de Janeiro, Departamento de An\'alise Matem\'atica, Instituto de Matem\'atica e Estat\'istica, Rio de Janeiro, Brasil}
\email{andre.furtado@ime.uerj.br}

\begin{document}

\maketitle

\begin{abstract}

By applying a Mawhin's continuation theorem of coincidence degree theory, we establish sufficient conditions for the existence of a periodic solution for a class of impulsive neutral differential equations. The procedure adopted in this work  makes use of a non-impulsive associated equation in order to overcome the difficulties resulting from the moments of impulse effects. 

\end{abstract}
\vspace{0.5cm}



\section{Introduction}
\hspace*{0.5cm}

Impulsive differential equations have gained an increasing importance in recent years due to their great suitability as a tool in modeling evolutionary processes {undergoing} sudden changes of states whose duration is negligible with respect to the duration of the whole process. These events occur in many areas of applied science such as mechanical systems with impact, population dynamics (when, for example, there are abrupt variations in population size), systems such as heart beats and blood flows, economics models, optimal control, and frequency modulation systems (see e.g.\ \cite{bainov1}, \cite{fu}, \cite{halanay}, \cite{Lakshmi} and \cite{Anatolii}). 
 
Neutral differential equations are a relevant particular type of functional differential equations. These equations apply to science and technology as, for instance, models in some variational problems, electrical networks containing lossless transmission lines and problems dealing with vibrating masses attached to an elastic bar (see e.g.\ \cite{Hale}).

{In \cite{zhao}, Jan and Zhao studied the oscillation of the solutions and the stability of zero solution of the first order linear delay impulsive differential equation }
\begin{equation}\label{00}
{\left\{
\begin{array}{ll}
x^\prime(t) + \displaystyle\sum_{i=1}^{n} p_i(t)x(t-\tau_i(t)) = 0 & \mbox{if}\  t\geq t_0,\ t\neq t_k\vspace{0,3cm}\\
x(t_k^+)-x(t_k)=b_k x(t_k)& \mbox{if}\ k=1,2,\ldots ,
\end{array}
\right.}
\end{equation}
{where $t_0\leq  t_1< t_2< \cdots < t_k<\cdots$, $t_k\to \infty$, $p_i:[t_0,\infty) \to \mathbb{R}$ are locally summable functions, $\tau_i:[t_0, \infty) \to [0,\infty)$ are Lebesgue measurable functions such that $t-\tau_i(t)\to \infty$ as $t\to \infty$ and $b_k\in (-\infty,-1)\cup (-1,\infty)$ are constant. 
They reduced the oscillation and nonoscillation of solutions of \eqref{00} and stability of
the zero solution of \eqref{00} to the corresponding problem, respectively, for a
delay differential equation without impulses. }

Recently, M. Li {\it et al} \cite{meili} applied the coincidence degree theory to investigate the existence of periodic solutions for the following impulsive differential problem with discrete delay:
\begin{equation*}
\left\{
\begin{array}{ll}
x^\prime(t)=f(t,x(t+\tau))& \mbox{if}\ \tau \in (-\infty, 0],\ t\geq 0,\ t\neq t_k\vspace{0,3cm}\\
x(t_k^+)-x(t_k)=b_k x(t_k)& \mbox{if}\ k=1,2,\ldots ,
\end{array}
\right.
\end{equation*}
where the $t_k$'s form a increasing and unbounded sequence of positive real numbers, $f$ is a real function defined on $(0,+\infty)\times \mathbb{R}$ and the $b_k$'s are real numbers larger than $-1$. {Making use of a non-impulsive equation conveniently associated with the above impulsive problem as in \cite{zhao}, the authors found sufficient conditions for the existence of periodic solutions.}

{Motivated by the techniques developed in \cite{zhao} and \cite{meili}}, in this paper we study the existence of a {T-periodic solution ($T>0$)} for the following impulsive neutral problem:
\begin{equation}\label{s1}
\left\{
\begin{array}{ll}
x^\prime(t) + Bx^\prime(t-\delta)=f(t,x_t)& \mbox{if}\ t\geq 0,\ t\neq t_1,t_2,\ldots\vspace{0,3cm}\\
x(t_k^+)-x(t_k)=b_k x(t_k)& \mbox{if}\ k=1,2,\ldots,
\end{array}
\right.
\end{equation}
where
\begin{itemize}
\item $f:[0,+\infty)\times G([-r,0],\mathbb{R}) \to \mathbb{R}$ is a $T$-periodic function on the first variable and such that, if $\varphi\in G([-r,T],\mathbb{R})$, then the function $t\mapsto f(t,\varphi_t)$, defined in $[0,T]$, lies in $G([0,T],\mathbb{R})$. Here, given an interval $[a,b]\subset \mathbb{R}$, the notation $G([a,b],\mathbb{R})$ denotes the Banach space of the regulated functions $x:[a,b]\to \mathbb {R}$ endowed with the supremum norm $\|x\|_{\infty}=\displaystyle\sup_{t\in[a,b]}|x(t)|$;\vspace{0,1cm}
\item $B$, $r$, and  $\delta$ are real numbers such that $0\leq \delta <  r\leq T$\vspace{0,1cm};
\item $t_1, t_2, \ldots$ are the {\it moments of impulse effects} of the problem and correspond to possible discontinuities of the solution and satisfies $0<t_1<\ldots<t_m<T-\delta$, $t_{m+k}=t_k+T$, $k=1,2,\ldots$\vspace{0,1cm};
\item $-1<b_1\ldots<b_m<T$, $b_{m+k}=b_k+T$, $k=1,2,\ldots$\vspace{0,1cm};
\item $x(t_k^+)$ represents the right limit $\lim_{t\to t_k^+}x(t)$\vspace{0,1cm};
\item Given a map $x:[-r,T]\to \mathbb{R}$ and $t\in[0,+\infty)$, the symbol $x_t$ denotes the function $x_t:[-r,0]\to\mathbb{R}$ given by $x_t(\tau)=x(t+\tau)${, for $\tau \in [-r,0]$}.
\end{itemize}

{Let us recall the concept of regulated functions, which can be found in \cite{didi}}. Given $\alpha, \beta\in \mathbb{R}, \alpha<\beta$, we say that a function $\phi:[\alpha,\beta]\to \mathbb{R}$ is {\it regulated} if

\[
\lim_{t\to \tau^-}\phi(t)\in \mathbb{R},\quad \mbox{for every}\ \tau\in (\alpha,\beta]
\]
and
\[
\lim_{t\to \tau^+}\phi(t)\in \mathbb{R},\quad \mbox{for every}\ \tau\in [\alpha,\beta).
\]

Let us define what we mean by a solution of problem (\ref{s1}).

\begin{definition}\label{definicao1neutra}
A function
$x:[-r,+\infty)\to\mathbb {R}$ is said to be a solution of problem (\ref{s1}) if
\begin{itemize}
\item[$(i)$] $x$ is absolutely continuous on each interval
$[-\delta,t_1],\ (t_k,t_{k+1}],\ k=1,2\ldots$;
  \item[$(ii)$] $x^\prime(t) + Bx^\prime(t-\delta)=f(t,x_t)$ almost everywhere in $[0,+\infty)\setminus\{t_1,t_2\ldots\}$;
\item[$(iii)$]$x(t_k^+)-x(t_k)=b_k x(t_k)$, $k=1,2,\ldots$.
\end{itemize}
\end{definition}
We say that $x$ is a $T$-periodic solution of (\ref{s1}) if $x$ satisfies the conditions given in Definition \ref{definicao1neutra} and is $T$-periodic on $[0,+\infty)$.

In order to find sufficient conditions for existence of a $T$-periodic solution for problem (\ref{s1})  it is {convenient to consider} the impulsive problem:
\begin{equation}\label{s3}
\left\{
\begin{array}{ll}
x^\prime(t) + Bx^\prime(t-\delta)=f(t,x_t)& \mbox{if}\ t\in[0,T]\setminus\{t_1,\ldots,t_m\}\vspace{0,3cm}\\
x(t_k^+)-x(t_k)=b_k x(t_k)& \mbox{if}\ k=1,\ldots,m,\vspace{0,3cm}\\
x(0)=x(T),
\end{array}
\right.
\end{equation}
where $B, \delta$ and $T$ are as in problem \eqref{s1}.

\begin{definition}\label{definicao3neutra}
A function
$x:[-r,T]\to\mathbb {R}$ is said to be a solution of problem (\ref{s3}) if
\begin{itemize}
\item[$(i)$] $x$ is absolutely continuous on each interval
$[-\delta,t_1],\ (t_1,t_2],\ldots,(t_m,T]$;
  \item[$(ii)$] $x^\prime(t) + Bx^\prime(t-\delta)=f(t,x_t)$ almost everywhere in $[0,T]\setminus\{t_1,\ldots,t_m\}$;
\item[$(iii)$]$x(t_k^+)-x(t_k)=b_k x(t_k)$ for $k\in \{1,\ldots m\}$;
\item[$(iv)$] $x(0)=x(T)$.
\end{itemize}
\end{definition}

This paper is organized as follows. In Section $2$ we present a non-impulsive problem appropriately associated with problem \eqref{s3} and establish useful relations between their possible solutions. In Section $3$ we obtain conditions ensuring the existence of a $T$-periodic solution of problem (\ref{s1}). Finally, in Section $4$, we present an example showing the effectiveness of the obtained result.

\section{Preliminaries}
\hspace*{0.5cm}

In what follows we present a non-impulsive problem for which the existence of a solution {implies} the same property for impulsive problem (\ref{s3}) and vice versa.

Define the functions $\beta: [-r,T]\to (0,+\infty)$ and $h:[0,T]\times G([-r,0],\mathbb{R})\to \mathbb{R}$ by
\begin{equation}\label{denitionbeta}
\beta(t)=\left\{
\begin{array}{lll}
\underset{t_k<t}\prod (1+b_k)& \mbox{if}\ t\in (t_1,T-\delta) \vspace{0,3cm}\\
1& \mbox{if}\ t\in [-r,t_1]\cup[T-\delta,T]\vspace{0,3cm}
\end{array}
\right.
\end{equation}
and
\[
h(t,\varphi)= \displaystyle\frac{f(t,\beta_t\varphi)}{\beta(t)}.
\]

Consider the neutral non-impulsive problem
\begin{equation}\label{s2}
\left\{
\begin{array}{ll}
u^\prime(t)+\dfrac{B\beta(t-\delta)}{\beta(t)}u^\prime(t-\delta)=h(t,u_t)& \mbox{if}\ t\in[0,T] \vspace{0,3cm}\\
u(0)=u(T),
\end{array}
\right.
\end{equation}
where $B, \delta$ and $T$ are as in problem \eqref{s1} and \eqref{s3}.

\begin{definition}\label{definicao2neutras}
A function
$u:[-r,T]\to\mathbb {R}$
is said to be a solution of problem (\ref{s2}) if
\begin{itemize}
\item[$(i)$]$u$ is absolutely continuous on the interval $[-\delta,T]$;
\item[$(ii)$]$u^\prime(t)+\displaystyle\frac{B\beta(t-\delta)u^\prime(t-\delta)}{\beta(t)}=h(t,u_t)$, almost everywhere in $[0,T]$.
\item[$(iii)$]u(0)=u(T)
\end{itemize}
\end{definition}

Next lemma establishes an important connection between the solutions of problems (\ref{s2}) and (\ref{s3}).

\begin{lemma}\label{l1}
Problem (\ref{s2}) has one solution if and only if the same occurs with problem (\ref{s3}).
\end{lemma}
\begin{proof}
Suppose that a function $u:[-r,T]\to \mathbb{R}$ is a solution of problem (\ref{s2}).
Consider the function $x(t)=\beta(t)u(t), t\in[-r,T]$. Since $u$ and $\beta$ are absolutely continuous on each interval $[-\delta,t_1],\ (t_1,t_2],\ldots,(t_m,T]$, the same is {valid} for $x$. That is, $x$ fulfills condition $(i)$ of Definition \ref{definicao3neutra}.

On the other hand, by the second item of Definition \ref{definicao2neutras}, we have
\[
x^\prime(t)+Bx^\prime(t-\delta)-f(t,x_t)
\]
\[
=\beta(t)u^\prime(t)+B\beta(t-\delta)u^\prime(t-\delta)-f(t,\beta_tu_t)
\]
\[=\beta(t)\left[u^\prime(t)+\frac{B\beta(t-\delta)u^\prime(t-\delta)}{\beta(t)}-\frac{f(t,\beta_tu_t)}{\beta(t)}\right]
\]
\[
=\beta(t)\left[u^\prime(t)+\frac{B\beta(t-\delta)u^\prime(t-\delta)}{\beta(t)}-h(t,u_t)\right]=0,
\]
for almost every $t\in [0,T]\setminus \{t_1,\ldots,t_m\}$. {Thus}, $x$ satisfies condition $(ii)$ of Definition \ref{definicao3neutra}.

Now, let us prove that $x$ verifies condition $(iii)$ of the above mentioned definition. For each $k\in \{1,\ldots m\}$ we have
\begin{equation*}
x(t_k^+)-x(t_k)= \lim_{t\to t_k^+}x(t)-x(t_k)=\lim_{t\to t_k^+}u(t)\beta(t)-x(t_k)=\lim_{t\to t_k^+}u(t)\prod_{t_j<t}(1+b_j)-x(t_k),
\end{equation*}
i.e.,
\begin{eqnarray*}
x(t_k^+)-x(t_k)&=&u(t_k)\prod_{t_j\leq t_k}(1+b_j)-x(t_k)=\frac{x(t_k)}{\beta(t_k)}(1+b_k)\prod_{t_j< t_k}(1+b_j)-x(t_k)\\\\
&=&\frac{x(t_k)}{\beta(t_k)}(1+b_k)\beta (t_k)-x(t_k)=b_kx(t_k).
\end{eqnarray*}

Finally, from the definition of $\beta$ and item $(iii)$ of Definition \ref{definicao2neutras}, it follows that $x(0)=x(T)$, that is, $x$ satisfies item $(iv)$ of Definition \ref{definicao3neutra}.

Let us prove that, if impulsive problem (\ref{s3}) has a solution, then non-impulsive problem (\ref{s2}) also has a solution. Suppose that $x:[-r,T]\to \mathbb{R}$ is a solution of problem (\ref{s3}). Consider the function $u(t)=x(t)/\beta(t), t\in [-r,T]$.
Since $x$ and $\beta$ are absolutely continuous  on the intervals $[-\delta,t_1],$ $(t_1,t_2]$, $\ldots$, $(t_m,T]$, then so is $u$. Thus, to conclude that $u$ satisfies item $(i)$ of Definition \ref{definicao2neutras} it is sufficient to prove that $u$ is right continuous on each moment of impulse $t_k$. In fact, for each $k=1,\ldots, m$, we have
\begin{eqnarray*}
\lim_{t\to t_k^+} u(t)&=&\lim_{t\to {t_k}^+}\frac{x(t)}{\beta(t)}=\lim_{t\to t_k^+}x(t)\prod_{t_j<t}(1+b_j)^{-1}
=x(t_k^+)\prod_{t_j\leq t_k}(1+b_j)^{-1}\\\\
&=&(1+b_k)x(t_k)(1+b_k)^{-1}\prod_{t_j<t}(1+b_j)^{-1}=\frac{x(t_k)}{\beta(t_k)}=u(t_k).
\end{eqnarray*}

Let us see now that $u$ satisfies condition $(ii)$ of Definition \ref{definicao2neutras}. Since the function $x$
is solution of problem (\ref{s3}), then
\begin{eqnarray*}
u^\prime(t)&=&\frac{x^\prime(t)}{\beta(t)}=\frac{f(t,x_t)-Bx^\prime(t-\delta)}{\beta(t)}\\\\
&=&\frac{f(t,\beta_tu_t)}{\beta(t)}-\frac{B\beta(t-\delta)u^\prime(t-\delta)}{\beta(t)}\\\\
&=&h(t,u_t)-\frac{B\beta(t-\delta)u^\prime(t-\delta)}{\beta(t)},
\end{eqnarray*}
for almost every $t\in[0,T]$.

Finally, from the definitions of $\beta$ and $x$ { it} follows that $u(0)=u(T)$, that is, $u$ satisfies condition $(iii)$ of Definition \ref{definicao2neutras}.
\end{proof}

\begin{remark}\label{remarkequivalencia}
By the demonstration of Lemma \ref{l1} we can infer that if non-impulsive problem (\ref {s2}) has a solution  on $[-r,T]$, then impulsive problem (\ref{s1}) has a $T$-periodic solution on $[-r,+\infty)$. In fact, if problem (\ref{s2}) has a solution on $[-r,T]$, then, by Lemma \ref {l1}, problem (\ref {s3}) has a solution $\hat{x}:[-r,T]\to \mathbb{R}$. Thus, the function $x:[-r,+\infty)\to \mathbb{R}$ given by $x(t)=\hat{x}(t)$ if $t\in [-r,T]$ and $x(t)=x(t-T)$ if $t\in(T,+\infty)$ is a $T$-periodic solution of impulsive problem (\ref{s1}).
\end{remark}

\section{Existence}

\hspace*{0.5cm}

{ Let us c}onsider the following {assumptions}:

\begin{enumerate}
\item[$(H1)$] If $\varphi:[-r,0]\to \mathbb{R}$ is absolutely continuous, then the function $t\mapsto f(t,\beta_t\varphi)$ is continuous on $[0,T]$;
\item[$(H2)$] There is a positive constant
$d$
such that, if
$\varphi \in G([-r,0],\mathbb{R})$
and
$|\varphi(0)|\geq d$,
then
\[
\varphi(0)f(t,\beta_t\varphi)>0,\quad t\in [0,T];
\]
\item[$(H3)$] $B<\displaystyle\frac{k}{K}$, where $k= \displaystyle\min_{-r\leq t\leq T}\beta(t)$ and $K=\displaystyle\max_{-r\leq t\leq T}\beta(t)$;
\item[$(H4)$]There is a positive constant $b<\displaystyle\frac{k-BK}{Tk}$ such that, if $\varphi,\psi\in G([-r,0],\mathbb{R})$, then
\[
|f(t,\varphi)-f(t,\psi)|\leq b|\varphi(0)-\psi(0)|,\quad t\in[0,T].
\]
\end{enumerate}

Our purpose in this section is to prove the following result.

\begin{theorem}\label{teoremaprincipalneutra}
If conditions $(H1)$ to $(H4)$ are satisfied, then problem (\ref{s1}) has at least one $T$-periodic solution.
\end{theorem}

To prove Theorem \ref{teoremaprincipalneutra} we actually show that conditions $(H1)$ to $(H4)$ imply the existence of a solution of problem \eqref{s2} on $[-r,T]$. Afterwards, Lemma \ref{l1} and Remark \ref{remarkequivalencia} complete our strategy. In order to do it,
 we use  a Mawhin's continuation theorem (Lemma \ref{teodemawhin} below), whose proof can be found in \cite{mawhin}, p. 40.
We need first  some basic concepts of functional analysis.

Let $X$ and $Y$ be Banach spaces and $W$ a subspace of $X$ with the induced norm. A {\it Fredholm operator} is a linear operator
$L:W \to Y$
whose range is a closed subspace of $Y$ and such that $\Ker L$ and $\coKer$ are finite dimensional. We define the {\it index} of $L$ by
$\dim \Ker L - \dim \coKer$.
It is well-known (see \cite{mawhin}) that, if
$L$ is a Fredholm operator of index zero, then there are linear, continuous and idempotent operators
$P:X\to X$
and
$Q:Y \to Y$
satisfying
\begin{equation}\label{KerimLneutra}
\Ker L = \Img P \hspace{1cm}\mbox{and}\hspace{1cm}\Img L = \Ker Q.
\end{equation}

The first expression in (\ref{KerimLneutra}) implies that the restriction of $L$ to $W \cap \Ker P$, $L_P:W \cap \Ker P\to L(W \cap \Ker P)$,
 is an isomorphism.

If
$L:W\to Y$
is a Fredholm operator of index zero and $P,\ Q$ are the operators mentioned above, then we say that a continuous operator
$N:X\to Y$
is {\it $L$-compact} in the closure
$\overline{\Omega}$
of a bounded open subset
$\Omega$ of $X$
if
$QN(\overline{\Omega})$
is bounded and
$(L_P)^{-1}(Id-Q)N: \overline{\Omega}\to X$
is a compact operator.

\begin{lemma}[Mawhin's Continuation Theorem]\label{teodemawhin}
Let
$X, Y$ be Banach spaces and $W$ a subspace of $X$ with the induced norm. Let
$L: W\to Y$ be a Fredholm operator of index zero and
$N: X\to Y$
a $L$-compact operator in
$\overline{\Omega}$,
where
$\Omega$
is an open bounded subset of
$X$.
Suppose, moreover, the following conditions hold:
\begin{itemize}
\item [($i$)]If $x\in \partial\Omega\cap W$ and $\lambda\in (0,1)$, then $Lx \neq \lambda Nx$ (here, $\partial\Omega$ denotes the boundary of $\Omega$ with respect to $X$);
\item [($ii$)]If $x\in \partial\Omega\cap \Ker L$, then $QNx \neq 0$;
\item [($iii$)] $\deg(JQN,\Omega \cap \Ker L, 0)\neq 0$, where $J:\Img Q\to \Ker L$ is any isomorphism, $JQN$ is defined on $\Ker L$ and $\deg(JQN,\Omega \cap \Ker L, 0)$ denotes the Brower degree of the triple $(JQN,\Omega \cap \Ker L, 0)$.
\end{itemize}
Under these conditions, the equation
$Lx=Nx$
has at least one solution in $\overline{\Omega}\cap W$.
\end{lemma}

Let $X= G([0,T],\mathbb{R})$ and let $W$ be the normed subspace of $X$ of the functions $x$
absolutely continuous such that $x(0)=x(T)$.
Define the operators $A:X\to X$, $L:W \to X$ and $N:X \to X$ by
\begin{equation}\label{definicaoA}
Ax(t)=x(t)+\frac{B\beta(t-\delta)x(t-\delta)}{\beta(t)},\quad t\in[0,T],
\end{equation}
\begin{equation}\label{definicaodeLeNneutra}
Lx = Ax^\prime,
\end{equation}
\begin{equation}\label{definicaodeNeNneutra}
Nx(t)=h(t,x_t),\quad t\in[0,T].
\end{equation}

\begin{remark}\label{convencao}
Due to the presence of  `$x(t-\delta)$' and `$x_t$' in the definitions of operators $A$ and $N$ we will adopt the convention $x(t)=x(t+T)$ for $t\in[-r,0)$, which will not affect our results.
\end{remark}

\begin{theorem}\label{teoremaexistenciasemimpulsoneutra}
If conditions $(H1)$ to $(H4)$ are fulfilled, then non-impulsive problem (\ref{s2}) has at least one solution  on $[-r,T]$.
\end{theorem}
The proof of Theorem \ref{teoremaexistenciasemimpulsoneutra}, which is essentially based on Lemma \ref{teodemawhin}, will be presented at the end of the section.

\begin{remark}\label{remarkoperadores}
Note that, if there is $\hat{x}\in W$ satisfying the equality $L\hat{x}=N\hat{x}$, then the function $x:[-r,T]\to \mathbb{R}$ given by $x(t)=\hat{x}(t)$ for $t\in[0,T]$ and $x(t)=\hat{x}(t+T)$ for $t\in[-r,0)$ is a solution for problem \eqref{s2} on $[-r,T]$.
\end{remark}

Thus, in order to demonstrate Theorem \ref{teoremaexistenciasemimpulsoneutra} we will show that conditions $(H1)$ to $(H4)$ imply  that the assumptions of Lemma \ref{teodemawhin} are satisfied when $L$ and $N$ are  defined as in (\ref{definicaodeLeNneutra}) and (\ref{definicaodeNeNneutra}).

Proposition \ref{lemaoperadorA} below shows an important property of $A$, based on the following classical result in Functional Analysis.

\begin{lemma}\label{lemaanalisefuncional}
Let $E$ be a Banach space and $F:E\to E$ a bounded linear operator such that $\|F\|<1$. Then $(Id-F)$ is a bijective operator and
\[
\|(Id-F)^{-1}\|\leq \frac{1}{1-\|F\|}.
\]
\end{lemma}

\begin{proposition}\label{lemaoperadorA}
If condition $(H3)$ holds, then the operator $A$, defined in (\ref{definicaoA}), is bijective and its inverse satisfies $\|A^{-1}\|\leq k/(k-BK)$ (where $k$ and $K$ are given in $(H3)$).
\end{proposition}
\begin{proof}
Let $F:X\to X$ be given by
\[
Fx(t)=-\frac{B\beta(t-\delta)x(t-\delta)}{\beta(t)},\quad t\in[0,T].
\]

For each $x\in X$ we have defined $x(t)=x(t+T)$, $t\in [-r,0]$. Thus, since $k= \displaystyle\min_{-r\leq t\leq T}\beta(t)$ and $K=\displaystyle\max_{-r\leq t\leq T}\beta(t)$ (see $(H3)$), we have
\begin{eqnarray*}
\|Fx\|&=&\sup_{t\in[0,T]}\frac{B\beta(t-\delta)|x(t-\delta)|}{\beta(t)}\\\\
&\leq& \frac{BK}{k}\sup_{t\in[0,T]}|x(t-\delta)|=\frac{BK}{k}\sup_{t\in[0,T]}|x(t)|=\frac{BK}{k}\|x\|_\infty.
\end{eqnarray*}
This and assumption $(H3)$ imply
\begin{equation*}
\|F\|\leq \frac{BK}{k} <1.
\end{equation*}
Then, by Lemma \ref{lemaanalisefuncional} we conclude that the operator $A=(Id-F):X\to X$ is bijective and
\[
\|A^{-1}\|=\|(Id-F)^{-1}\|\leq \frac{1}{1-\|F\|}\leq \frac{k}{k-BK}
\]
and the proof is complete.
\end{proof}

Now we can prove the following:
\begin{proposition}\label{Ledefredholmneutras}
If condition $(H3)$ holds, then the operator $L$, defined in (\ref{definicaodeLeNneutra}), is Fredholm of index zero.
\end{proposition}

\begin{proof}
Suppose that condition $(H3)$ is satisfied.
Let us start by showing that
\begin{equation}\label{imagemdeLneutras}
\Img L=\left\{y\in X;\ \int_0^T y(t)dt=0\right\}.
\end{equation}

Let $y\in X$ be such that $y=Lx$, for some $x\in W$. Then,
\begin{eqnarray*}
\int_0^T y(t)dt&=&\int_0^T Lx(t)dt=\int_0^T Ax^\prime (t)dt =\int_0^T \left[x^\prime(t)+\frac{B\beta(t-\delta)}{\beta(t)}x^\prime(t-\delta)\right]dt\\\\
&=&\int_0^T x^\prime(t)dt+B\int_0^T\frac{\beta(t-\delta)}{\beta(t)}x^\prime(t-\delta)dt\\\\
&=&[x(T)-x(0)]+B\left[\frac{\beta(T-\delta)x(T-\delta)}{\beta(T)}-\frac{\beta(-\delta)x(-\delta)}{\beta(0)}\right]\\\\
&=&0,
\end{eqnarray*}
where the last equality follows from the condition $x(0)=x(T)$, the definition of $\beta$ and Remark \ref{convencao}.
Thus,
\begin{equation}\label{imagemdeLneutras0}
\Img L \subset \left\{y\in X;\ \int_0^T y(t)dt=0\right\}.
\end{equation}

Let $y\in X$ be such that $\int_0^T y(t)dt=0$. Consider $x\in W$ given by
\[
x(t)=A^{-1}\left(\int_0^t y(s)ds\right),\quad t\in[0,T].
\]
Then
\[
Lx(t)= Ax^\prime(t)=\frac{d}{dt}\left[Ax(t)\right]=\frac{d}{dt}\left[\int_0^t y(s)ds\right]=y(t), \quad t\in[0,T].
\]
From this and (\ref{imagemdeLneutras0}) we get (\ref{imagemdeLneutras}).

Let us see now that $\Img L$ is closed in $X$. Let $\{y_n\}$ be a sequence in $\Img L$ converging to an element $y\in X$. Then,
\[
\int_0^Ty(t)dt=\int_0^T\lim_{n\to \infty}y_n(t)dt=\lim_{n\to \infty}\int_0^T y_n(t)dt=0,
\]
that is, $y\in \Img L$.

To finish the proof, it remains to show that $\dimm \Ker L =\codimm \Img L<+\infty$.
We begin by showing that
$\dim \Ker L=1$. In fact, if $x\in \Ker L$, then $0=Lx=Ax^\prime$ and thereby $x^\prime =0$. Therefore, since $x$ is absolutely continuous, $x$ is a constant function and thus $\dim \Ker L=1$.
Finally, we show that $\codimm \Img L=1$. Define the linear operator
\begin{equation}\label{definicaodeQneutras}
Qx=\frac{1}{T}\int_0^T x(t)dt,\quad x\in X.
\end{equation}
Then
\[
Q(Qx)=\frac{1}{T}\int_0^T Qx(t)dt=\frac{1}{T}\int_0^T\left[\frac{1}{T}\int_0^T x(\tau)d\tau\right]dt=\frac{1}{T}\int_0^T x(\tau)d\tau=Qx, \quad \textrm{for}\, x\in X.
\]
This shows that $Q$ is an idempotent operator. Moreover, $\Img Q \cap \Ker Q=\{0\}$. In fact, if $y\in \Img Q$, then there exists $x\in X$ such that
\[
y=\frac{1}{T}\int_0^T x(t)dt,
\]
that is, $y$ is a constant function and hence, belonging to $\Ker Q$, is identically zero. Thus, $\Img Q \cap \Ker Q=\{0\}$, and then $X=\Ker Q \,\oplus\, \Img Q$, since $Q$ is linear and idempotent. Thus, $\codimm \Ker Q=1$, since $\Img Q \simeq \mathbb{R}$. Moreover, by (\ref{imagemdeLneutras}), $\Img L =\Ker Q$ and this complete the proof.
\end{proof}

We show here that there is a bounded and open set $\Omega \subset X$ such that $N$ is $L$-compact in the closure of $\Omega$ and such that conditions $(i)$, $(ii)$ and $(iii)$ of Lemma \ref{teodemawhin} hold. In the sequel, in order to find $\Omega$, we obtain an {\it a priori} estimate for the equation
\begin{equation}\label{equacaolambda}
Lx(t)=\lambda Nx(t),\quad t\in[0,T],
\end{equation}
where $\lambda\in(0,1)$ is fixed.

\begin{proposition}\label{lGi}
Suppose that conditions $(H1)$ and $(H2)$ hold.
If
$x\in W$ satisfies (\ref{equacaolambda}), then
\[
\|x\|_\infty \leq d+\sqrt{T}\|x^\prime\|_2,
\]
where $\|x^\prime\|_2=\left(\displaystyle\int_0^T x^\prime(t)^2dt\right)^{1/2}$ and $d$ is the constant present in assumption $(H2)$.
\end{proposition}
\begin{proof}
Assume that conditions $(H1)$ and $(H2)$ are fulfilled and that there is $x\in W$ satisfying (\ref{equacaolambda}).
Integrating (\ref{equacaolambda}) on $[0,T]$, we obtain
\begin{equation}\label{integrandoambosneutras}
\int_0^T\left[x^\prime(t)+\frac{B\beta(t-\delta)x^\prime(t-\delta)}{\beta(t)}\right]dt=\lambda \int_0^Th(t,x_t)dt.
\end{equation}
From Remark \ref{convencao}, the definition of $\beta$ and from the equality $x(0)=x(T)$, we infer that the left   side of (\ref{integrandoambosneutras}) is zero. Thereby,
\begin{equation*}
\int_0^T h(t,x_t)dt=0,
\end{equation*}
that is,
\begin{equation}\label{intdeh0}
\int_0^T \frac{f(t,\beta_tx_t)}{\beta(t)}dt=0.
\end{equation}
Since
$\beta$
is integrable, positive and, from $(H1)$, the function $t\mapsto f(t,\beta_tx_t)$ is continuous on $[0,T]$, then, from (\ref{intdeh0}) it follows that there exists
$\tau \in [0,T]$
such that
\begin{equation}\label{eqtau}
f(\tau,\beta_\tau x_\tau)=0.
\end{equation}
This implies
$|x(\tau)|<d$.
In fact, if
$|x(\tau)|\geq d$,
then
$|x_\tau (0)|\geq d$
and therefore, from assumption $(H2)$, $x_\tau (0)f(\tau,\beta_\tau x_\tau)>0,$
which contradicts (\ref{eqtau}). Accordingly,
$|x(\tau)|<d$. This and the triangle and H\"{o}lder inequalities imply
\[
|x(t)|=\left|x(\tau)+\int_\tau^t x^\prime (s)ds\right|< d + \int_0^T \left|x^\prime (s)\right|ds \leq d + \sqrt{T}\left\|x^\prime\right\|_2,\quad t\in [0,T],
\]
and the proof is complete.
\end{proof}

\begin{proposition}\label{lema2.5}
If conditions $(H1)$ to $(H4)$ hold, then there is a constant
$D>0$, with $D$ independent of $\lambda$,
such that, if
$x\in W$ satisfies (\ref{equacaolambda}), then $\|x\|_\infty \leq D$.
 \end{proposition}
\begin{proof}
Suppose that conditions $(H1)$ to $(H4)$ are satisfied and that there is $x\in W$ verifying (\ref{equacaolambda}). Multiplying (\ref{equacaolambda}) by
$x^\prime(t)$, we obtain
\[
[x^\prime(t)]^2+ \frac{B\beta(t-\delta)x^\prime(t)x^\prime(t-\delta)}{\beta(t)}= \lambda h(t,x_t)x^\prime(t),\quad t\in[0,T].
\]
Using the definition of $h$, the triangle inequality and   $k\leq \beta(t)\leq K$, $-r\leq t\leq T$ (see assumption $(H3)$), we obtain
\begin{eqnarray*}
[x^\prime(t)]^2&<&\frac{|f(t,\beta_tx_t)||x^\prime(t)|}{\beta(t)}+\frac{BK|x^\prime(t)||x^\prime(t-\delta)|}{k}\\\\
&\leq&\frac{|f(t,\beta_tx_t)-f(t,0)||x^\prime(t)|}{\beta(t)}+\frac{|f(t,0)||x^\prime(t)|+BK|x^\prime(t)||x^\prime(t-\delta)|}{k},\quad t\in[0,T].
\end{eqnarray*}
This and assumption $(H4)$ imply
\[
[x^\prime(t)]^2\leq b|x(t)||x^\prime(t)|+\frac{|f(t,0)||x^\prime(t)|+BK|x^\prime(t)||x^\prime(t-\delta)|}{k},\quad t\in[0,T]
\]
and, thus
\[
[x^\prime(t)]^2\leq b\|x\|_\infty|x^\prime(t)|+\frac{|f(t,0)||x^\prime(t)|+BK|x^\prime(t)||x^\prime(t-\delta)|}{k},\quad t\in[0,T].
\]

Integrating both sides of the last inequality on $[0,T]$ and using the H\"{o}lder inequality, we obtain
\begin{eqnarray*}
\|x^\prime\|_2^2 \leq \sqrt{T}b\|x\|_\infty\|x^\prime\|_2+\frac{S\sqrt{T}}{k}\|x^\prime\|_2+\frac{BK}{k}\|x^\prime\|_2^2,
\end{eqnarray*}
where $|f(t,0)|\leq S$  for every $t\in[0,T]$. This and Proposition \ref{lGi} give us
\begin{eqnarray*}
\|x^\prime\|_2^2&\leq&bT\|x^\prime\|_2^2+\left(bd+\frac{S}{k}\right)\sqrt{T}\|x^\prime\|_2+\frac{BK}{k}\|x^\prime\|_2^2.
\end{eqnarray*}
Moreover, by assumption $(H4)$, $b<\displaystyle\frac{k-BK}{Tk}$. Thus, we get
\[
\|x^\prime\|_2 \leq \frac{(bdk+S)\sqrt{T}}{k-bTk-BK},
\]
which, by Proposition \ref{lGi}, implies
\[
\|x\|_\infty \leq D = d+\frac{(bdk+S)T}{k-bTk-BK}.
\]
This completes the proof.
\end{proof}

Now we are ready to introduce the following proposition.

\begin{proposition}\label{lemaqnomegalimitadoneutra}
If conditions $(H1)$ to $(H4)$ hold, then there is an open and bounded set $\Omega \subset X$ such that the operator $N$ is $L$-compact in $\overline{\Omega}$.
\end{proposition}
\begin{proof}
Suppose that $(H1)$ to $(H4)$ are satisfied. By Proposition \ref{lema2.5}, there exists a constant $D>0$, which does not depend on $\lambda$, such that  $\|x\|_\infty \leq D$ for any $x\in W$ verifying (\ref{equacaolambda}). Let $M>D$ be given and consider the set
\begin{equation}\label{definicaodeomega}
\Omega=\left\{x\in W;\ \|x\|_\infty < M\right\}.
\end{equation}
Below we will show that the operator $N$ is $L$-compact in $\overline{\Omega}$.

Define the operators $P:X\to X$ and $Q:X\to X$ by
\[
Px=x(0)\qquad {\rm and} \qquad Qx=\frac{1}{T}\int_0^T x(t)dt.
\]

Then $\Img P=\Ker L $ and, by (\ref{imagemdeLneutras}), $\Ker Q=\Img L$.

We shall show that the set $QN({\overline{\Omega}})$ is bounded and the operator $L_P^{-1}(Id-Q)N:\overline{\Omega}\to X$ is compact, where $L_P^{-1}$
denote the inverse of the restriction of $L$ to $\Ker P$.

We begin by showing that $QN({\overline{\Omega}})$ is bounded. For each $x\in \overline{\Omega}$, we have
\begin{equation}\label{QNlimitadoneutras}
\|QNx\|_\infty=\frac{1}{T}\left|\int_0^T Nx(t)dt\right|= \frac{1}{T}\left|\int_0^T \frac{f(t,\beta_tx_t)}{\beta(t)}dt\right|\leq \frac{1}{T} \int_0^T \frac{|f(t,\beta_tx_t)|}{\beta(t)}dt.
\end{equation}

On the other hand, by assumption $(H1)$, there is $R>0$ such that $|f(t,\beta_tx_t|\leq R$, $t\in [0,T]$. As, in addiction, by assumption $(H3)$, $k\leq \beta(t)$, $t\in[0,T]$, then from (\ref{QNlimitadoneutras}) it follows that
\[
\|QNx\|_\infty \leq \frac{R}{k}.
\]

In the sequel, we show that the operator $L_P^{-1}(Id-Q)N:\overline{\Omega}\to X$ is compact. Let
$y\in L(W \cap \Ker P)$ be given and consider the function $x\in W\cap \Ker P$ defined by $x(t)=\displaystyle\int_0^t A^{-1}y(s)ds$. Then
\[
L x(t)=Ax^\prime(t)=AA^{-1}y(t)=y(t),\quad t\in[0,T].
\]

Therefore,
\begin{equation*}
L_P^{-1}y(t)=\int_0^t A^{-1}y(s)ds,\quad t\in[0,T].
\end{equation*}
Thus,
\begin{equation}\label{inversadeLpneutra}
L_P^{-1}(Id-Q)Nx(t)=\int_0^t A^{-1}(Id-Q)Nx(s)ds,\quad t\in[0,T].
\end{equation}

Let $\Lambda$ be a  subset of $\overline{\Omega}$ and
let $\{x^n\}$ be
a sequence of functions in ${\Lambda}$. In order to show that $L_P^{-1}(Id-Q)N$ is compact we will prove that there is a subsequence $\{x^{n_k}\}\subset \{x^n\}$ such that
$\{L_P^{-1}(Id-Q)N x^{n_k}\}$ is convergent. By Arzel\`a-Ascoli's Theorem is sufficient to show that the sequence of functions
$\{L_P^{-1}(Id-Q)N x^n\}$
is equicontinuous and uniformly bounded.

In the sequel, for the sake of simplicity, we will denote
$F_n=L_P^{-1}(Id-Q)N x^n$. We begin by showing that
$\{F_n\}$
is equicontinuous. Given
$t_0 \in[0,T]$ and $\epsilon >0$, let us show that there is
$\delta=\delta(t_0,\epsilon)>0$
such that
\begin{equation}\label{equicontinuidadedefinicaoneutra}
t \in (t_0-\delta,t_0+\delta)\cap [0,T] \Longrightarrow |F_n(t)-F_n(t_0)|<\epsilon,
\end{equation}
for any
$n\in \mathbb{N}$.

By (\ref{inversadeLpneutra}), we conclude that for each $n\in \mathbb{N}$ one has
\begin{eqnarray}\label{equicontinua}
|F_n(t)-F_n(t_0)|&=& \left|\int_0^t A^{-1}(Id-Q)Nx^n(s)ds-\int_0^{t_0} A^{-1}(Id-Q)Nx^n(s)ds\right|\nonumber\\\nonumber\\
&=&\left|\int_{t_0}^t A^{-1}(Id-Q)Nx^n(s)ds\right|\nonumber\\\nonumber\\
&\leq& \int_{t_0}^t |A^{-1}(Id-Q)Nx^n(s)|ds,\quad t\in[0,T].
\end{eqnarray}

On the other hand, by definitions of operators $N$ and $Q$ and the triangle inequality, we obtain
\begin{eqnarray*}
|(Id-Q)Nx(s)|&=&\left|\frac{f(s,\beta_sx_s)}{\beta(s)}
-\frac{1}{T}\int_0^T\frac{f(s,\beta_sx_s)}{\beta(s)}\right|\\\\
&\leq&\frac{2R}{k},\quad s\in[0,T].
\end{eqnarray*}
This and Proposition \ref{lemaoperadorA} imply
\begin{equation}\label{caso1b}
|A^{-1}(Id-Q)Nx^n(s)|\leq \|A^{-1}\|\frac{2R}{k}\leq \frac{2R}{k-BK},\quad s\in[0,T].
\end{equation}

Now, expression (\ref{equicontinua}) implies
\[
|F_n(t)-F_n(t_0)|\leq\frac{2R|t-t_0|}{k-BK},\quad t\in[0,T].
\]
Thus, taking $\delta< \epsilon (k-BK)/2R$ we obtain expression (\ref{equicontinuidadedefinicaoneutra}) and, therefore, the sequence of functions $\{F_n\}$ is equicontinuous.

Furthermore,
$\{F_n\}$
is uniformly bounded. In effect, by (\ref{caso1b}) and Proposition \ref{lemaoperadorA}, we obtain
\begin{eqnarray*}
|F_n(t)|&=&\left|\int_0^t A^{-1}(Id-Q)Nx^n(s)ds\right|\leq \int_0^{T} |A^{-1}(Id-Q)Nx^n(s)|ds\\\\
&\leq& T\|A^{-1}\|\|(Id-Q)Nx^n\|_\infty \leq \frac{2TR}{k-BK},\quad t\in[0,T],
\end{eqnarray*}
for every $n\in \mathbb{N}$.
\end{proof}

Let us now show that, if $(H1)$ to $(H4)$ hold, then the assumptions of Lemma \ref{teodemawhin} are satisfied.

\begin{proposition}\label{teomawhinaplicado}
Suppose that conditions $(H1)$ to $(H4)$ are fulfilled. Let $\Omega$ be the set defined in (\ref{definicaodeomega}) and $L$, $N$ and $Q$ the operators defined in (\ref{definicaodeLeNneutra}), (\ref{definicaodeNeNneutra}) and (\ref{definicaodeQneutras}). Then, we have:
\begin{itemize}
\item [($i$)]$x\in \partial\Omega\cap W \Longrightarrow Lx \neq \lambda Nx,\ \lambda\in (0,1)$;
\item [($ii$)]$x\in \partial\Omega\cap \Ker L \Longrightarrow QNx \neq 0$;
\item [($iii$)]$\deg(JQN,\Omega \cap \Ker L, 0)\neq 0$, where $\deg(JQN,\Omega \cap \Ker L, 0)$ is as in item $(iii)$ of Lemma \ref{teodemawhin}.
\end{itemize}
\end{proposition}
\begin{proof}
Since
$M>D$,
then, by Proposition \ref{lema2.5}, $Lx \neq \lambda Nx$,
for any
$x\in \partial\Omega\cap W$
and
$\lambda\in(0,1)$, that is, the condition
$(i)$
holds.

Now we prove $(ii)$. Let
$x\in \partial\Omega\cap \Ker L$ be given. By definition of $\Omega$ and since $\Ker L\simeq \mathbb{R}$, then $x\equiv M$ or $x\equiv -M$. In the first case,
$x_t(0)=M>D>d,\ t\in [0,T]$, which, by $(H2)$, implies
\begin{eqnarray*}
f(t,\beta_t x_t)>0,\quad t\in [0,T].
\end{eqnarray*}
Therefore,
\[
\int_0^T Nx(t)dt=\int_0^T \frac{f(t,\beta_t x_t)}{\beta(t)}dt >0.
\]

Analogously, if $x\equiv -M$, then $\int_0^T Nx(t)dt<0$. Hence, item $(ii)$ holds.

To end the proof, let us show that condition $(iii)$ holds. Since $\Img Q$ and $\Ker L$ can be identified, we can take $J$ as the identity operator $J:\Img Q \to \Ker L $, $J(x)=x$. Thus, we shall show that $\deg(QN,\Omega \cap \Ker L, 0)\neq 0$. In order to do it we use the homotopy invariance of the degree.

Set
$H: (\overline{\Omega}\cap \Ker L)\times [0,1]\to \mathbb {R}$ by
\[H(x,\lambda)=(1-\lambda)x + \frac{\lambda}{T}\int_0^T Nx(t)dt.
\]

Since
$\Omega\cap \Ker L=(-M,M)$, then
\[\partial (\Omega \cap \Ker L) \times [0,1]= \{-M,M\}\times [0,1].
\]

Now we prove that
$0\not \in H(\{-M,M\}\times [0,1])$.
For every $\lambda \in [0,1]$, we have
\begin{equation}\label{eq20}
H(M,\lambda)=(1-\lambda)M + \frac{\lambda}{T}\int_0^T \frac{f(t,\beta_t M)}{\beta(t)}dt.
\end{equation}
Note that the first term on the right member of (\ref{eq20}) is non-negative. Let us see that the same holds for the second term. Since $M>d$, then, by the assumption $(H2)$, we have
\[
f(t,\beta_t M)>0,\quad t\in [0,T],
\]
and thus,
\begin{equation}\label{segundaparcelanaonegativa}
\frac{\lambda}{T}\int_0^T \frac{f(t,\beta_t M)}{\beta(t)}dt \geq 0.
\end{equation}
Moreover, note  that we have the equality in (\ref{segundaparcelanaonegativa}) only when $\lambda =0$, but in this case the first term in the right member of (\ref{eq20}) is positive. On the other hand, the first term in the right member of (\ref{eq20}) vanishes only when $\lambda =1$, and in this case the second term is positive.
Therefore, $H(M,\lambda)>0$ for every $\lambda \in [0,1]$. Similarly, $H(-M,\lambda)<0$ for every $\lambda \in [0,1]$. This shows that $0\notin H(\partial(\Omega\cap \Ker L)\times [0,1])$, which means that $\deg(H(\cdot, 0),\Omega \cap \Ker L,0)$ and $\deg(H(\cdot, 1),\Omega \cap \Ker L,0)$ are well defined.

By homotopy invariance of degree, we have
\begin{equation}\label{invariancia}
\deg(H(\cdot, 0),\Omega \cap \Ker L,0)=\deg(H(\cdot,1),\Omega \cap \Ker L,0).
\end{equation}
Furthermore, for any $x\in \Omega \cap \Ker L$, we have
\begin{equation}\label{hdex1hdex0}
H(x,0)= x\qquad \mbox{and}\qquad H(x,1)= \frac{1}{T}\int_0^T Nx(t)dt= QNx(t).
\end{equation}
By (\ref{invariancia}), (\ref{hdex1hdex0}) and the normalization property of the degree, we obtain
\begin{eqnarray*}
\deg(QN,\Omega \cap \Ker L,0)&=&\deg(H(\cdot,1),\Omega \cap \Ker L,0)=\deg(H(\cdot,0),\Omega \cap \Ker L,0)\\
&=&\deg(Id,\Omega \cap \Ker L,0)=1,
\end{eqnarray*}
which finishes the proof.
\end{proof}

\begin{proof}
(of Theorem \ref{teoremaexistenciasemimpulsoneutra}) Assume that $(H1)$ to $(H4)$ hold. By Lemma \ref{teodemawhin} and Proposition \ref{teomawhinaplicado}, there is at least one $x\in W$ such that $Lx=Nx$. Then, by Remark \ref{remarkoperadores}, problem (\ref{s2}) has at least one solution on $[-r,T]$.
\end{proof}

Finally, we can conclude our process.
\begin{proof}
(of Theorem \ref{teoremaprincipalneutra})
The result follows immediately from Theorem \ref{teoremaexistenciasemimpulsoneutra} and Remark \ref{remarkequivalencia}.
\end{proof}

\section{Example}

In problem (\ref{s1}) take $r=2$, $\delta=1$, $T=2\pi$, $B=1/13$, $m=2$ and the numbers $b_k$ and the moments of impulse $t_k$ as
$b_1=2,\ b_2=3, t_1=1,\ t_2=3/2$, $t_{m+k}=t_k+T$, $b_{m+k}=b_k+T$, $k=1,2,\ldots$. In this way, the function $\beta$, defined in (\ref{denitionbeta}), takes the form $\beta:[-2,2\pi]\to(0,+\infty)$, with
\begin{equation*}
\beta(t)=\left\{
\begin{array}{ll}
1& \mbox{if}\ t\in[-2,1]\cup[2\pi-1,2\pi]\vspace{0,2cm}\\
3& \mbox{if}\ t\in(1,3/2]\vspace{0,2cm}\\
12& \mbox{if}\ t\in(3/2,2\pi-1).
\end{array}
\right.
\end{equation*}

Consider the problem
\begin{equation}\label{problemaexemplo}
\left\{
\begin{array}{ll}
x^\prime(t) + \displaystyle\frac{1}{13}x^\prime(t-\delta)=f(t,x_t)& \mbox{if}\ t\geq 0,\ t\neq t_1,t_2,\ldots\vspace{0,3cm}\\
x(t_k^+)-x(t_k)=b_k& k=1,2,\ldots,
\end{array}
\right.
\end{equation}
where the function $f:[0,+\infty)\times G([-2,0],\mathbb{R})\to \mathbb{R}$ is given by
$f(t,\varphi)=\displaystyle\frac{\varphi(0)|\cos t|}{27\pi \beta(t)}$.

Let us see that, in this case, conditions $(H1)$ to $(H4)$ are satisfied.
\begin{itemize}
\item [$(H1)$:] Let $\varphi:[-2,0]\to \mathbb{R}$ be an absolutely continuous function. Clearly, the function $t\mapsto f(t,\beta_t\varphi)=\displaystyle\frac{\varphi(0)|\cos t|}{27\pi }$ is continuous on $[0,2\pi]$.
\item [$(H2)$:] Let $\varphi$ be a function in $G([-2,0],\mathbb{R})$, such that $|\varphi(0)|\geq 1$. Then $\varphi(0)f(t,\beta_t\varphi)=\displaystyle\frac{\varphi(0)^2|\cos t|}{27\pi }>0$ for each $t\in[0,2\pi]$.
\item [$(H3)$:] Obviously, in this case we have $B<k/K$, where $k= \displaystyle\min_{-2\leq t\leq 2\pi}\beta(t)$ and $K=\displaystyle\max_{-2\leq t\leq 2\pi}\beta(t)$, since $B=1/13$, $k=1$ and $K=12$.
\item [$(H4)$:] Let $\varphi, \psi$ be two functions in $G([-2,0],\mathbb{R})$. Then, for each $t\in[0,2\pi]$, we obtain
\[
|f(t,\varphi)-f(t,\psi)|=\frac{|\cos t|}{27\pi \beta(t)}|\varphi(0)- \psi(0)|\leq b|\varphi(0)- \psi(0)|,
\]
where $b=\displaystyle\frac{1}{27\pi}<\displaystyle\frac{1}{26\pi}=\displaystyle\frac{k-BK}{Tk}$.
\end{itemize}

By Theorem \ref{teoremaprincipalneutra}, problem (\ref{problemaexemplo}) has {at least one $2\pi$-periodic solution on $[0,+\infty)$}.

{\section*{Acknowledgment}}

{The authors would like to thank the anonymous referees for their valuable comments and suggestions which helped to improve the manuscript.}

\end{document}